\documentclass[11pt]{amsart}
\usepackage{amsmath, amsthm, amssymb}
\usepackage{graphicx}
\usepackage[usenames]{color}
\usepackage{srcltx} 
\usepackage{verbatim}
\usepackage{tikz-cd}
\usepackage{mathtools}

\newtheorem{thm}{Theorem}[section]
\newcommand{\bt}{\begin{thm}}
\newcommand{\et}{\end{thm}}

\newtheorem{ex}[thm]{Example}

\newtheorem{cor}[thm]{Corollary}   
\newcommand{\bc}{\begin{cor}}
\newcommand{\ec}{\end{cor}}

\newtheorem{lem}[thm]{Lemma}   
\newcommand{\bl}{\begin{lem}}
\newcommand{\el}{\end{lem}}

\newtheorem{prop}[thm]{Proposition}
\newcommand{\bp}{\begin{prop}}
\newcommand{\ep}{\end{prop}}

\newtheorem{defn}[thm]{Definition}
\newcommand{\bd}{\begin{defn}}    
\newcommand{\ed}{\end{defn}}

\newtheorem{rmrk}[thm]{Remark}   

\newcommand{\br}{\begin{rmrk}}
\newcommand{\er}{\end{rmrk}}

\newcommand{\mGHto}{\stackrel { \textrm{mGH}}{\longrightarrow} }
\newcommand{\GHto}{\stackrel { \textrm{GH}}{\longrightarrow} }
\newcommand{\Fto}{\stackrel {\mathcal{F}}{\longrightarrow} }
\newcommand{\VFto}{\stackrel {\mathcal{VF}}{\longrightarrow} }

\newcommand{\be}{\begin{equation}}

\newcommand{\ee}{\end{equation}}

\newcommand{\N}{\mathbb{N}}

\newcommand{\R}{\mathbb{R}}

\DeclareMathOperator{\inj}{inj}
\DeclareMathOperator{\expon}{exp}

\newcommand{\diam}{\operatorname{Diam}}

\newcommand{\vol}{\operatorname{Vol}}

\newcommand{\lp}{\left (}
\newcommand{\rp}{\right )}

\newcommand{\Sp}{\mathbb{S}}      
\newcommand{\Tor}{\mathbb{T}}     

\def\Xint#1{\mathchoice
{\XXint\displaystyle\textstyle{#1}}%
{\XXint\textstyle\scriptstyle{#1}}%
{\XXint\scriptstyle\scriptscriptstyle{#1}}%
{\XXint\scriptscriptstyle\scriptscriptstyle{#1}}%
\!\int}
\def\XXint#1#2#3{{\setbox0=\hbox{$#1{#2#3}{\int}$ }
\vcenter{\hbox{$#2#3$ }}\kern-.6\wd0}}

\def\dashint{\Xint-}

\begin{document}

\title[Gromov-Hausdorff Convergence and $L^p$ Bounds]{From $L^p$ Bounds to Gromov-Hausdorff Convergence of Riemannian Manifolds}

\author{Brian Allen}
\address{University of Hartford}
\email{brianallenmath@gmail.com}





\begin{abstract}
In this paper we provide a way of taking $L^p$, $p > \frac{m}{2}$ bounds on a $m-$ dimensional Riemannian metric and transforming that into H\"{o}lder bounds for the corresponding distance function. One can think of this new estimate as a type of Morrey inequality for Riemannian manifolds where one thinks of a Riemannian metric as the gradient of the corresponding distance function so that the $L^p$, $p > \frac{m}{2}$ bound analogously implies H\"{o}lder control on the distance function. This new estimate is then used to state a compactness theorem, another theorem which guarantees convergence to a particular Riemmanian manifold, and a new scalar torus stability result. We expect these results to be useful for proving geometric stability results in the presence of scalar curvature bounds when Gromov-Hausdorff convergence is expected. 
\end{abstract}

\maketitle

\section{Introduction}

Gromov's compactness theorem \cite{Gromov-poly, Gromov-metric} provides geometric conditions on a sequence of metric spaces which guarantees a subsequence converges in the Gromov-Hausdorff sense. The advantage of this compactness theorem is that the hypotheses are natural from a metric geometry point of view. One just needs a bound on diameter and for every $\varepsilon > 0$ a bound on the number of $\varepsilon$-balls which cover the metric space. In this paper we are interested in compactness and convergence theorems which assume natural geometric analysis assumptions on a sequence of Riemannian manifolds. In particular, we are interested in understanding for which $p>1$ does an $L^p$ bound on a Riemannian manifold imply Gromov-Hausdorff convergence.

 When attempting to prove geometric stability results for Riemannian manifolds it can be useful to obtain weaker estimates first, such as $L^p$ estimates for the Riemannian metric, and then use these estimates to bootstrap up to a stronger notion of convergence. In this paper we provide a way of taking $L^p$, $p > \frac{m}{2}$ bounds on an $m-$dimensional Riemannian metric and transforming that into H\"{o}lder bounds for the corresponding distance function. One can think of this new estimate as a type of Morrey inequality for Riemannian manifolds where one thinks of a Riemannian metric as the gradient of the corresponding distance function so that the $L^p$, $p > \frac{m}{2}$ bound analogously implies H\"{o}lder control on the distance function. We begin by stating the H\"{o}lder estimate on the distance function which is used to prove the subsequent results and which is interesting in its own right.

\begin{thm}\label{MainThmEst}
Let $M^m$ be a smooth, closed,  connected manifold, $M_0=(M,g_0)$ a smooth Riemannian manifold, and $M_1=(M,g_1)$ a continuous Riemannian manifold. If
 \begin{align}
 \exists p > m, \quad \|g_1\|_{L_{g_0}^{\frac{p}{2}}(M)} \le C
 \end{align}
  then 
\begin{align}
d_1(q_1,q_2) \le C'(K,m,p,I_0,D_0,C) d_0(q_1,q_2)^{\frac{p-m}{p}},\quad \forall q_1,q_2 \in M,
\end{align}
where $K$ is a bound on the absolute value of the sectional curvature of $M_0$, $I_0=\inj(M,g_0)$ is the injectivity radius of $M_0$, and $D_0=\diam(M_0)$ is the diameter of $M_0$.
\end{thm}

\begin{rmrk}
We note that the sectional curvature, injectivity radius, and diameter bounds are only needed on the background Riemannian manifold and not on the Riemannian manifold whose distance function is being estimated.  The proof uses a construction where a family of curves connecting two points $p,q \in M$, centered around a distance minimizing curve, is constructed which foliates a region of the background Riemannian manifold $M_0$. These geometric assumptions on $M_0$ are to control this foliated region.
\end{rmrk}

There is a rich history of using elliptic regularity to bootstrap up from $L^p$ bounds on Riemannian metrics in coordinates to $C^{k,\alpha}$ and $W^{k,p}$ bounds on Riemannian metrics including the work of Anderson \cite{Anderson-Ricci, Anderson-Orbifold}, Anderson and Cheeger \cite{Anderson-Cheeger}, Cheeger and Colding \cite{Cheeger-Colding-1}, Colding \cite{Colding-shape, Colding-volume}, Gao \cite{Gao-integral1}, Petersen and Wei \cite{Petersen-Wei-integral1, Petersen-Wei-integral2}, and Yang \cite{DYang-integral1, DYang-integral2, DYang-integral3}(See the survey by Petersen \cite{Petersen-Survey} for a broad overview). An important fact used in most of those works (in addition to many other important estimates) is a bound on Ricci curvature (integral or pointwise) and the crucial result that one can view Ricci curvature as a elliptic PDE for the corresponding Riemannian metric when taking advantage of harmonic coordinates. This is what allows one to bootstrap up to H\"{o}lder or Sobolev control on the Riemannian metric which we note is generally stronger than H\"{o}lder control on the distance function. When a bound on Ricci curvature is not appropriate one should not expect to obtain H\"{o}lder control on the Riemannian metric and what is interesting about the current results is that one can still obtain H\"{o}lder control on the distance function if an $L^{\frac{p}{2}}$, $p>m$ bound on the Riemannian metric is assumed.

In the work of Aldana, Carron, and Tapie \cite{Aldana-Carron-Tapie} a similar estimate is observed for conformal metrics (See Proposition 2.2). In the conformal case the H\"{o}lder compactness for the distance functions is a simple consequence of a relationship between the gradient of the distance function and the conformal factor. They are then able to use this result to show interesting compactness results in the presence of integral bounds on scalar curvature as well as results on $A_{\infty}$ weights. In the present work the importance of Theorem \ref{MainThmEst} is that it holds for general Riemannian metrics and because of this the result requires a completely different proof than in the conformal case. 

In the work of Bryden and the author \cite{Allen-Bryden} Sobolev bounds on a Riemannian metric and are able to obtain H\"{o}lder control on the corresponding distance function by taking advantage of new trace inequalities. These trace inequalities  allow the authors to transfer the Sobolev control on the Riemannian metric to control of integrals of the metric along curves which we then relate to distances. The advantage of the current results is that they only require $L^p$, $p > \frac{m}{2}$ bounds on the metric but should also be seen as complementary to the results in \cite{Allen-Bryden}.

We then use the new estimate in Theorem \ref{MainThmEst} to prove a compactness result and a convergence result which we state as follows. One advantage of this compactness theorem is that you also gain H\"{o}lder control on the distance function of the limiting metric space.

\begin{thm}\label{MainThmComp}
Let $M^m$ be a smooth, closed, connected manifold $M_0=(M,g_0)$ a smooth Riemannian manifold and $M_j=(M,g_j)$ a sequence of continuous Riemannian manifolds with corresponding distance functions $d_0, d_j$, respectively. If
 \begin{align}
 \exists p > m, \quad \|g_j\|_{L_{g_0}^{\frac{p}{2}}(M)}  \le C\label{ConvergenceLmNormGH}
 \end{align}
 then there exists a function $d_{\infty}:M\times M \rightarrow [0,\infty)$ so that a subsequence $d_k \rightarrow d_{\infty}$ uniformly as functions,
   \begin{align}\label{LimitingHolderEst1}
d_{\infty}(q_1,q_2) \le C'(K,m,p,I_0,D_0, C)d_0(q_1,q_2)^{\frac{p-1}{p}},\quad \forall q_1,q_2 \in M,
  \end{align}
  where $K$ is a bound on the absolute value of the sectional curvature of $M_0$, $I_0=\inj(M,g_0)$ is the injectivity radius of $M_0$, and $D_0=\diam(M_0)$ is the diameter of $M_0$.  Furthermore, if we quotient $M$ by points whose $d_{\infty}$ distance is zero, $M_{\infty}=(M/d_{\infty},d_{\infty})$,  we find Gromov-Hausdorff convergence
\begin{align}
 M_k\GHto M_{\infty}.
\end{align}
 
 If in addition
\begin{align}
\exists c > 0, \eta > 1,\quad c d_0(q_1,q_2)^{\eta} \le d_j(q_1,q_2), \quad \forall q_1,q_2\in M,  \label{metricbounds}
\end{align}
  then there exists a metric $d_{\infty}$ so that
  \begin{align}\label{LimitingHolderEst}
 c d_0(q_1,q_2)^{\eta} &\le d_{\infty}(q_1,q_2) 
 \\&\le C'(K,m,p,I_0,D_0,C)d_0(q_1,q_2)^{\frac{p-1}{p}}, \forall q_1,q_2 \in M,
  \end{align} 
 such that for $M_{\infty}=(M,d_{\infty})$ we find Gromov-Hausdorff convergence
\begin{align}
 M_k\GHto M_{\infty}.
\end{align}
\end{thm}

We now see that if we require a $L^{\frac{p}{2}}$ bound for $p > m$, volume convergence, and continuous convergence from below we can ensure that the limiting metric space is the particular background Riemannian manifold $M_0$. One should review Example 3.1, which first appeared in the work of the author and Sormani \cite{Allen-Sormani-2}, which shows that without the $C^0$ convergence from below that this theorem cannot be true. This example is one of a general family of examples which shows that $C^0$ convergence from below is the right condition to combine with $L^p$ convergence (or volume convergence) to imply Gromov-Hausdorff or Sormani-Wenger intrinsic flat convergence. The following theorem is related to the result of Perales, Sormani, and the author \cite{Allen-Perales-Sormani} where if one does not assume a $L^{\frac{p}{2}}$ bound for $p > m$ then one obtains just volume preserving Sormani-Wenger intrinsic flat convergence of the sequence.

\begin{thm}\label{MainThmConv}
Let $M^m$ be a smooth, closed, oriented, connected manifold $M_0=(M,g_0)$ a smooth Riemannian manifold and $M_j=(M,g_j)$ a sequence of continuous Riemannian manifolds. Then if
\begin{align}
\lp 1 - 1/j\rp g_0(v,v) \le g_j(v,v),\quad \forall p \in M, v \in T_pM,\label{metricbounds}
\end{align}
\begin{align}
\diam(M_j) \le D_0,
\end{align}
 \begin{align}
 \exists p > m, \quad \|g_j\|_{L_{g_0}^{\frac{p}{2}}(M)}\le C, \label{LmNormBound}
 \end{align}
 and 
 \begin{align}
 \vol(M_j) \rightarrow \vol(M_0)\label{VolumeConv}
 \end{align}
  then 
\begin{align}
M_j \VFto M_0,
\\ M_j\mGHto M_0.
\end{align}
\end{thm}

\begin{rmrk}
 In particular, this theorem shows that when combined with $C^0$ convergence from below, the dimension is the correct threshold for $L^{\frac{p}{2}}$ bounds to distinguish between when one should expect Sormani-Wenger intrinsic flat convergence only and when one should expect that convergence as well as Gromov-Hausdorff convergence. When one only observes an $L^{\frac{m}{2}}$ bound one should expect bubbling to occur along the sequence (as in Example 3.3) so that the sequence will not converge to a Riemannian manifold in either sense. When one observes $L^{\frac{m}{2}}$ convergence (or volume convergence) one should only expect Sormani-Wenger intrinsic flat convergence to the Riemannian manifold $M_0$ (as in Example 3.4). When one observes $L^{\frac{p}{2}}$, $p > m$ convergence (or volume convergence combined with an  $L^{\frac{p}{2}}$, $p > m$ bound) then one should expect Gromov-Hausdorff and Sormani-Wenger intrinsic flat convergence to a Riemannian manifold as in Theorem \ref{MainThmConv} and Example 3.2. These important examples are reviewed in section \ref{sect: examples} which were originally given by the author and Sormani in \cite{Allen-Sormani-2}. 
 
 We lastly note that one could replace \eqref{LmNormBound} and \eqref{VolumeConv} with $L^{\frac{p}{2}}$, $p > m$ convergence of norm for the sequence of Riemannian metrics and obtain the same conclusion.
\end{rmrk}

In \cite{AHPPW} by Hernandez, Parise, Payne, Shengwang, and the author, as well as, \cite{Allen-Tori} by the author, special cases of Gromov's conjecture on tori with almost non-negative scalar curvature were solved. In both cases, $L^p$ bounds on a sequence of metrics was obtained first by studying the scalar curvature PDE for warped products and conformal metrics, respectively. Then, the maximum principle and mean value inequality were used to obtain the necessary $C^0$ bound from below to apply the main theorem of Sormani and the author in \cite{Allen-Sormani, Allen-Sormani-2} and \cite{Allen-Perales-Sormani}, respectively.  Hence the theorem of this paper is a type of analogue of those results designed to be applied in a similar way in a case where Gromov-Hausdorff convergence is expected and $L^{\frac{p}{2}}$, $p > m$ bounds are naturally obtained or assumed for the metric. 

In \cite{Allen-Tori}, the author proves a version of Gromov's conjecture on tori with almost non-negative scalar curvature \cite{GroD} in the conformal case. If one replaces the uniform volume bound (equation (3) of Theorem 1.2 and equation (12) of Theorem 1.5) and the uniform volume bound on balls (equation (6) of Theorem 1.2 and equation (16) of Theorem 1.5) of the main theorems of \cite{Allen-Tori} with a $L^{\frac{p}{2}}$, $p > m$ bound on the metric (or equivalently a $L^p$, $p > m$ bound on the conformal factor) one can immediately conclude measured Gromov-Hausdorff convergence by applying Theorem \ref{MainThmConv}, as we state precisely below. This shows an immediate application of the main theorem of this paper to problems involving scalar curvature. This also shows that one should expect measured Gromov-Hausdorff convergence in Gromov's conjecture if the stronger assumption of an $L^{\frac{p}{2}}$, $p > m$ bound on the metric is assumed. 

\begin{thm}\label{TorusTheorem}
Let $g_0$ be a flat torus where $\mathbb{T}^m_0=(\mathbb{T}^m,g_0)$. For a sequence of Riemannian $m-$manifolds $M_j=(\mathbb{T},g_j)$, $m \ge 3$ satisfying   \begin{align} \label{HypothesisMainThm2}
R_{g_j} \ge -\frac{1}{j},    \,\,\,  \,\,\, \|g_j\|_{L_{g_0}^{\frac{p}{2}}(M)} \le V_0, \,\, p > m,
\end{align}
where $R_{g_j}$ is the scalar curvature of $M_j$ and so that $M_j$ is conformal to $\tilde{M}_{0,j}=(\mathbb{T}^m,\tilde{g}_{0,j})$, a metric with constant zero or negative scalar curvature and unit volume, i.e.  $g_j = e^{2f_j} g_{0,j}$. Furthermore, assume that
\begin{align}
\tilde{g}_{0,j} \rightarrow g_0 \text{ in } C^1,
\end{align}
and
\begin{align}
\int_{\mathbb{T}^m} e^{-2f_j} d V_{\tilde{g}_{0,j}} \le C,
\end{align} 
then there exists a subsequence so that $M_k$ converges in the volume preserving intrinsic flat sense and the measured Gromov-Hausdorff sense to a flat torus
\begin{align}
M_k &\VFto \bar{\mathbb{T}}_0^m,
\\ M_k &\mGHto \bar{\mathbb{T}}_0^m,
\end{align}
where $\bar{\mathbb{T}}_0^m = (\Tor^m,\bar{g}_0=c_{\infty}^2g_0)$, $c_{\infty}^2=\displaystyle\lim_{k \rightarrow \infty}(\overline{e^{-f_k}})^{-2}= \lim_{k \rightarrow \infty} \left( \dashint_{\mathbb{T}^m} e^{-f_k}dV_{\tilde{g}_{0,j}} \right)^{-2}$.
\end{thm}

In section \ref{sect: review} we review important definitions and theorems which will be used throughout this paper.

In section \ref{sect: examples} we review examples  of Sormani and the author \cite{Allen-Sormani-2} which explains the intuition for why one should expect an $L^{\frac{p}{2}}$, $p > m$ bound to imply a H\"{o}lder distance bound as well as why the $C^0$ bound from below on the metric is necessary in Theorem \ref{MainThmConv}.

In section \ref{sec:MainThmProof} we introduce the important construction of a symmetric family of curves between $q_1$ and $q_2$ of radius $\varepsilon$ and use this construction to prove the main estimate of this paper. We finish by using this H\"{o}lder estimate in order to prove the main compactness and convergence theorem as well as a version of the geometric stability of the scalar torus rigidity theorem.\\

\noindent{\bf Acknowledgements:} This research was funded in part by NSF DMS - 1612049. 

\section{Background} \label{sect: review}
In this section we review important definitions and theorems which will be used throughout the paper. 
\subsection{Gromov-Hausdorff Distance}

If we have a complete metric space $(Z,d_Z)$ and subsets $X_1, X_2 \subset Z$ then we can define the Hausdorff distance between them to be
\begin{align}
d_H^Z(X_1,X_2)=\inf\{\varepsilon > 0: X_1\subset T_{\varepsilon}(X_2) \text{ and }X_2\subset T_{\varepsilon}(X_1) \},
\end{align}
where $T_{\varepsilon}(X_i)=\{y \in Z: d(y,X_i)\le \varepsilon\}$, $i=1,2$. 

Now if one considers two metric spaces $(Y_1,d_{Y_1})$ and $(Y_2,d_{Y_2})$ we say that $\varphi_i: Y_i\rightarrow Z$, $i=1,2$ is a distance preserving map if
\begin{align}
d_{Y_i}(q_1,q_2)= d_Z(\varphi_i(q_1),\varphi_i(q_2)),\quad \forall q_1,q_2 \in Y_i.
\end{align}
We can now define the Gromov-Hausdorff distance between $(Y_1,d_{Y_1})$ and $(Y_2,d_{Y_2})$ to be
\begin{align}
d_{GH}((Y_1,d_{Y_1}),(Y_2&,d_{Y_2}))=\inf\{d_H^Z(\varphi_1(Y_1),\varphi_2(Y_2)):
\\&Z \text{ complete }, \varphi_i:Y_i\rightarrow Z \text{ distance preserving}\}.
\end{align}

We say that a sequence of metric spaces $(Y_i,d_{Y_i})$ converges in the Gromov-Hausdorff sense to a metric space $(Y_{\infty},d_{\infty})$ if
\begin{align}
d_{GH}((Y_i,d_{Y_i}),(Y_{\infty},d_{\infty})) \rightarrow 0.
\end{align}

For a sequence of Riemannian manifolds $M_i=(M,g_i)$ which are Gromov-Hausdorff converging to a Riemannian manifold $M_{\infty}=(M,g_{\infty})$ we say that they converge in the measured Gromov-Hausdorff sense if we also have that
\begin{align}
\vol_{g_j}(B_{g_j}(p,r))\rightarrow \vol_{g_{\infty}}(B_{g_{\infty}}(p,r)),\quad \forall p \in M, r > 0.
\end{align}

Gromov then showed that this distance is a true distance on compact metric spaces. Gromov's compactness theorem then says that if the diameter of a sequence of metric spaces is bounded and for each $r > 0$ the maximum number of disjoint balls of radius $r$  is uniformly bounded over the sequence then a subsequence exists which Gromov-Hausdorff converges to a metric space. In the case where the sequence of metric spaces are all defined on the same set $X$ and H\"{o}lder bounds are known one can obtain compactness from the standard Arzella-Ascolli theorem applied to the distance functions which we now state formally so that we can apply it later to prove Theorem \ref{MainThmComp}.

\begin{thm}\label{Arzella-Ascolli}
Let $X_{\alpha}=(X,d_{\alpha})$, $\alpha \in \N \cup \{0\}$ be metric spaces so that there exists a $c, C >0$, $\alpha \ge 1$, and $0<\beta \le 1$ so that
\begin{align}
cd_0(q_1,q_2)^{\alpha} \le d_j(q_1,q_2) \le C d_0(q_1,q_2)^{\beta}, \quad \forall q_1,q_2 \in M
\end{align}
then there exists a metric $d_{\infty}$ and a metric space $X_{\infty}=(X,d_{\infty})$ so that
\begin{align}
X_j \GHto X_{\infty},
\end{align}
where 
\begin{align}
cd_0(q_1,q_2)^{\alpha} \le d_{\infty}(q_1,q_2) \le C d_0(q_1,q_2)^{\beta}, \quad \forall q_1,q_2 \in X.
\end{align}
\end{thm}

One can compare this to the compactness theorem in the appendix of \cite{HLS} where Lipshchitz bounds are assumed and GH and Sormani-Wenger Intrinsic Flat (SWIF) convergence is obtained. Also, in the appendix of \cite{Allen-Bryden} where a Lipschitz bound from below and H\"{o}lder bound from above are assumed and GH and SWIF convergence is obtained.

\subsection{Lebesgue Norm of Riemannian Metrics}

When looking for estimates on sequences of Riemannian manifolds which are weaker then GH distance one is naturally lead to consider $L^p$ notions of distance between Riemannian manifolds which we now review. If $g_0,g_1$ are two Riemannian metrics defined on the manifold $M$ then we can define the norm of $g_1$ with respect to $g_0$ in coordinates to be
\begin{align}
|g_1|_{g_0}=\sqrt{(g_0)^{ij}(g_0)^{lm} (g_1)_{il}(g_1)_{jm}}.
\end{align}
This allows us to define the $L^p$ norm of $g_1$ with respect to $g_0$ to be
\begin{align}
\|g_1\|_{L^p_{g_0}(M)}=\left(\int_M  |g_1|_{g_0}^p dV_{g_0}\right)^{\frac{1}{p}},
\end{align}
where $dV_{g_0}$ is the volume form for $g_0$. We say that a sequence of Riemannian metrics defined on $M$, $g_j$, converges to $g_{\infty}$ in $L^p_{g_0}(M)$ if
\begin{align*}
\|g_j-g_{\infty}\|_{L^p_{g_0}(M)} \rightarrow 0,
\end{align*}
and we say that $g_j$ converges to $g_{\infty}$ in $L^p$ norm if
\begin{align*}
\|g_j\|_{L^p_{g_0}(M)} \rightarrow \|g_{\infty}\|_{L^p_{g_0}(M)}.
\end{align*}
See subsection 2.6 of \cite{Allen-Sormani-2} for some standard results on $L^p$ convergence.

We now include a proof of a simple lemma which is useful for comparing norms of two different Riemannian metrics.

\begin{lem}\label{NormComparisonLemma}
 For $v \in T_pM$ and $g_0,g_1$ Riemannian metrics on $M$ we have the inequality
 \begin{align}
 |v|_{g_1}\le |g_1|_{g_0}^{\frac{1}{2}}|v|_{g_0}.
 \end{align}
 \end{lem}
 \begin{proof}
 Notice by the Cauchy-Schwarz inequality
 \begin{align}
 |g_0(g_1,dv \otimes dv)| \le |g_1|_{g_0} |dv \otimes dv |_{g_0} = |g_1|_{g_0} |v|_{g_0}^2,
 \end{align}
 and then we find
  \begin{align}
 g_0(g_1,dv \otimes dv) &=(g_0)^{ij}(g_0)^{pq} (g_1)_{ip}(dv \otimes dv)_{jq} 
 \\&= (g_0)^{ij}(g_0)^{pq} (g_1)_{ip}v_j  v_q
 \\&= (g_1)^{jq}v_jv_q = |v|_{g_1}^2,
 \end{align}
 which gives the desired result. 
 \end{proof}

\subsection{Review of Volume Above Distance Below}

The Perales, Sormani, and the author in \cite{Allen-Perales-Sormani} showed that if one has natural geometric assumptions on a sequence of Riemannian manifolds then one can guarantee volume preserving intrinsic flat convergence to a particular Riemannian manifold. It is by combining the following theorem with Theorem \ref{MainThmEst} which allows us to conclude Theorem \ref{MainThmConv}. Many examples which justify the hypotheses of this main theorem were given as warped products in \cite{Allen-Sormani} and conformal metrics in \cite{Allen-Sormani-2}. Important examples to the intuition of this paper are reviewed in section \ref{sect: examples}.

\begin{thm}\label{VADB} 
Suppose we have a fixed  compact, connected, oriented Riemannian manifold, $M_0=(M^n,g_0)$,
without boundary and
a sequence of Riemannian manifolds $M_j=(M, g_j)$ with
\be 
g_0(v,v) \le g_j(v,v) \qquad \forall p \in M, v\in T_pM
\ee 
and a uniform upper bound on diameter
\be
\diam_j(M_j) \le D_0
\ee
and volume convergence
\be
\vol(M_j) \to \vol(M_0)
\ee
then $M_j$ converge to $M_0$ in the volume preserving intrinsic flat sense
\be
M_j \VFto M_0.
\ee
\end{thm}

We note that convergence in $L^{\frac{p}{2}}$ norm of $g_j$ to $g_0$ combined with the $C^0$ convergence from below implies the required convergence of volume assumed in Theorem \ref{VADB}. See Lemma 4.3 of \cite{Allen-Sormani-2} for a proof of this result.

\section{Examples of Sequences of Conformal Manifolds}\label{sect: examples}

Here we review some examples which first appeared in the work of the author and Sormani \cite{Allen-Sormani-2}. Every example is conformal to the torus or the sphere and we note that $L^{\frac{p}{2}}$ convergence of Riemannian metrics, $g_j=f_j^2g_0$, is equivalent to $L^p$ convergence of the conformal factors $f_j$. To summarize the main intuition to take away from these examples we have included Table 1 and Figure 1. 

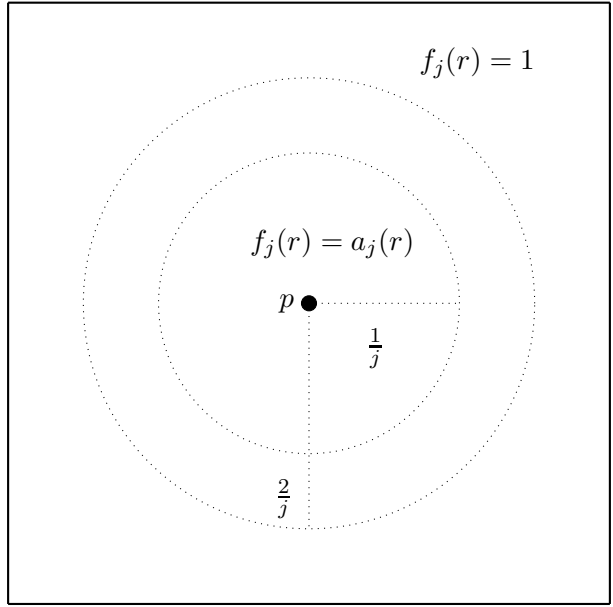
\begin{figure}\label{TorusExPic}
 \label{fig:SymmetricFamilyCurves}
 \begin{tikzpicture}[scale=4]
 \draw[thick] (-1,1) -- (1,1);
 \draw[thick] (-1,-1) -- (-1,1);
 \draw[thick] (1,-1) -- (-1,-1);
 \draw[thick] (1,1) -- (1,-1);
  \draw[dotted] (0,0) -- (.5,0);
  \draw[dotted] (0,0) -- (0,-.75);
  \node[left, outer sep=2pt] at (0,0) {$p$};
   \node[left, outer sep=2pt] at (.8,.8) {$f_j(r)=1$};
   \node[left, outer sep=2pt] at (.4,.2) {$f_j(r)=a_j(r)$};
   \node[left, outer sep=2pt] at (.3,-.15) {$\frac{1}{j}$};
   \node[left, outer sep=2pt] at (-.005,-.65) {$\frac{2}{j}$};
  \draw[dotted] (0,0) circle [radius=.75];
  \draw[dotted] (0,0) circle [radius=.5];
  \draw[fill] (0,0) circle [ radius=.025];
 \end{tikzpicture} 
  \caption{Depiction of a flat torus with a conformal factor $f_j(r)$ defined radially from $p$ which is different than $1$ on a ball of radius $\frac{2}{j}$ and strictly increasing on the annulus where $\frac{1}{j} < r < \frac{2}{j}$.}
 \end{figure}

\begin{table}[ht] \label{ExamplesTable}
\caption{\textbf{Examples Contrasting $L^p$, GH, and SWIF Convergence:} We see that when one has $L^p$ convergence for $p > m$ then we find H\"{o}lder control on distances and all three notions of convergence agree. When one has $L^m$ convergence then SWIF convergence agrees with $L^p$ convergence but GH convergence differs. When one has $L^p$ convergence for $p < m$ and an $L^m$ bound then one expects bubbling and when there is no $L^m$ bound then one should expect poor behavior of the sequence with respect to GH and SWIF convergence. In particular, the example sequence converges to a flat torus with an infinite cigar attached. See Figure 1.}
\begin{tabular}{c c c c c}  
\hline\hline                        
Cases & $a_j(r)=j^{\alpha}$ & $a_j(r)=\frac{j^{\eta}}{1+\ln(j)}, 0 < r < \frac{1}{j^{\eta}},
$  & $a_j(r) =j$ & $a_j(r)=j^{\eta}$ \\ 
& $ 0<\alpha < 1$ & $\frac{1}{r(1-\ln(r))}, \frac{1}{j^{\eta}} < r < \frac{1}{j}, \eta > 1$ & & $\eta > 1$\\[1ex]
\hline 
 &&&&\\
$L^p$ Converges & $p<\frac{m}{\alpha} > m$ & $p \le m$ & $p < m$,  & $p < \frac{m}{\eta} < m$  \\    
&&&$L^m$ is bounded&\\[1ex]

GH converges  & Yes & No, converges to a  & No, converges to a   & No, unbounded  \\
to flat tori? & & flat torus with a  &flat torus with a  & volume   \\ 
 & &  line attached & bubble attached &and diameter  \\ [1ex] 
SWIF converges  & Yes & Yes  & No, converges to a   & No, unbounded    \\ 
to flat tori?& & &flat torus with a  &volume \\
& & &bubble attached &and diameter \\
Prototypical  &H\"{o}lder & & &  \\ 
geometric & control & splines & bubbling & blowing up \\ 
phenomenon & on distances & & &  \\ [1ex] 

\hline     
\end{tabular}

\end{table}


The first example shows that the $C^0$ bound from below is necessary for proving convergence of $M_j$ to $M_0$ in Theorem \ref{MainThmConv}. One should note that this example can be generalized to a family of examples where the conformal factor converges pointwise, but not uniformly, along  a curve to a value which is lower than what the conformal factor converges to on the rest of the manifold. For this family of examples one will always find a cinched metric space in the limit where the curve produces a shortcut in the limiting metric space and the limiting metric space cannot be a Riemannian manifold.

 \begin{ex} \label{Cinched-Sphere}  
 Define a sequence of functions radially from the north pole on $\Sp^m$ by
 \be
 f_j(r)=
 \begin{cases}
 1 & r\in[0,\pi/2- 1/j]
 \\  h(jr-\pi/2) & r\in[\pi/2- 1/j, \pi/2+ 1/j]
 \\ 1 &r\in [\pi/2+ 1/j, \pi]
 \end{cases}
\ee
where $h:[-1,1]\rightarrow \R$ is a smooth even function such that 
$h(-1)=1$ with $h'(-1)=0$, 
decreasing to $h(0)=h_0\in (0,1)$ and then
increasing back up to $h(1)=1$, $h'(1)=0$. If one defines $M_j = (\Sp^m, f_j^2 g_{\Sp^m})$ we see that
\begin{align}
M_j &\VFto M_{\infty}
\\M_j &\mGHto M_{\infty}
\end{align}
but we can conclude that $M_{\infty}$ is not isometric to $\Sp^m$. Instead $M_{\infty} = (\Sp^m, f_{\infty}^2 g_{\Sp^m}) $ is the conformal metric with conformal factor
 \be
 f_{\infty}(r)=
 \begin{cases}
 h_0 & r=\pi/2
 \\  1 &\text{ otherwise}
 \end{cases},
 \ee  
 which defines a metric space which is not a Riemmanian manifold.
\end{ex}

In the next example we see a simple case where the hypotheses of Theorem \ref{MainThmConv} are satisfied and the conclusion holds.

\begin{ex}\label{L^p Conv}
Define a sequence of functions, radially defined from a point $p \in \Tor^m$, on $\Tor^m$ by 
\begin{equation}
f_j(r)=
\begin{cases}
j^{\alpha} &\text{ if } r \in [0,1/j]
\\h_j(jr) &\text{ if } r \in [1/j,2/j]
\\ 1 & \text{ if } r \in (1/j,\sqrt{2}\pi]
\end{cases}
\end{equation}
where $0< \alpha < 1$ and $h_j:[1,2] \rightarrow \R$ is a smooth, decreasing function so that $h_j(1) = j^{\alpha}$ and $h_j(2) = 1$. Then $\|f_j-1\|_{L^p}\rightarrow 0$, $p< \frac{m}{\alpha} $ and 
\begin{align}
M_j &\VFto \Tor^m
\\M_j &\GHto \Tor^m.
\end{align}
\end{ex}

The next example is not bounded in $L^p$ for any $p > m$ and does not converge in $L^m$ to $1$ and so does not fit the hypotheses of Theorem \ref{MainThmConv}. For this reason we will see that the Gromov-Hausdorff and Sormani-Wenger intrinsic flat limit are not the flat torus due to bubbling. Instead the limit is a flat torus with a bubble attached. This is an important example since it shows that when the Riemannian metric is bounded in $L^{\frac{m}{2}}$ one should expect the possibility of bubbling.

\begin{ex}\label{No L^m Conv}
Define a sequence of functions, radially defined from a point $p \in \Tor^m$, on $\Tor^m$ by
\begin{equation}
f_j(r)=
\begin{cases}
j &\text{ if } r \in [0,1/j]
\\h_j(jr) &\text{ if } r \in [1/j,2/j]
\\ 1 & \text{ if } r \in (2/j,\sqrt{m}\pi].
\end{cases}
\end{equation}
where $h_j:[1,2] \rightarrow \R$ is a smooth, decreasing function so that $h_j(1) = j$, $h_j'(1)=h_j'(2)=0$, and $h_j(2) = 1$ so that
\begin{align}
    \frac{1}{j^m}\int_1^2h_j(s)^m s^{m-1}ds \rightarrow 0.\label{ConstructionHyp}
\end{align}
Then $f_j$ is not bounded in $L^p$ norm for $p > m$ but does have bounded $L^m$ norm and volume.
Furthermore, for $M_j=(\Tor^m, f_j^2 g_{\Tor^m})$
\begin{align}
M_j &\Fto M_{\infty}
\\M_j &\GHto M_{\infty}
\end{align}
where $M_{\infty}$ is not isometric to $\Tor^m$. Instead 
\begin{align}
    M_{\infty} = \Tor^m \sqcup \mathbb{D}^m/\sim,
\end{align}
where we fix a $p \in \Tor^m$ and for $d \in \partial \mathbb{D}^m$  we have
\begin{align}
    p \sim d.
\end{align}
\end{ex}

In the next example $f_j\rightarrow 1$ in $L^m$ but the $L^p$ limit is unbounded for every $ p > m$ which is at the boundary of what is assumed in Theorem \ref{MainThmConv}.  For this reason we will see that the Sormani-Wenger intrinsic flat limit and the Gromov-Hausdorff limit disagree. This example, combined with the insight of Example 3.2 and 3.3, illustrates that the $\frac{m}{2}$ power for the $L^p$ convergence of Riemannian metrics is the correct power which distinguishes between Gromov-Hausdorff convergence and Sormani-Wenger intrinsic flat convergence.

\begin{ex}\label{VolControlDiamNotConvergent}
Define a sequence of functions, radially defined from a point $p \in \Tor^m$, on $\Tor^m$ by
\begin{equation}
f_j(r)=
\begin{cases}
\frac{j^{\eta}}{1+\ln(j)} &\text{ if } r \in [0,1/j^{\eta}]
\\ \frac{1}{r(1-\ln(r))} &\text{ if } r \in (1/j^{\eta},1/j]
\\ h_j(jr) &\text{ if } r \in (1/j,2/j]
\\ 1 & \text{ if } r \in (2/j,\sqrt{m}\pi].
\end{cases}
\end{equation}
where $\eta > 1$ and $h_j:[1,2]\rightarrow \R$ is a smooth, decreasing function so that $h_j(1) = \frac{j}{1+\ln(j)}$  and $h_j(2) = 1$. $f_j$ is unbounded in $L^p$, $p > m$, but
\begin{align}
    \|f_j-1\|_{L^p(\Tor^m)} \rightarrow 0 \text{ for } p \le m.
\end{align}
If one defined $M_j = (\Tor^m, f_j^2 g_{\Tor^m})$ then
\begin{align}
M_j &\mGHto M_{\infty},
\end{align}
where $M_{\infty}$ is $\Tor^m$ with a line of length $\ln(\eta)$ attached, and
\begin{align}
M_j &\VFto \Tor^m.
\end{align}
\end{ex}

In conclusion, we have seen that $C^0$ convergence from below is necessary to expect convergence to a Riemannian manifold in Theorem \ref{MainThmConv} and if one additionally assumes $L^{\frac{p}{2}}$, $p > m$ convergence then one expects Gromov-Hausdorff and Sormani-Wenger intrinsic flat convergence to a Riemannian manifold. If one assumes $L^{\frac{p}{2}}$, $p = m$ convergence then one only expects  Sormani-Wenger intrinsic flat convergence to a Riemannian manifold, as in Corollary 5.2 of \cite{Allen-Perales-Sormani}. In the case of Theorem 1.1 of \cite{Allen-Perales-Sormani} we see that convergence of volume is enough and we don't have to assume $L^{\frac{m}{2}}$ convergence of the metric. In the case of Theorem \ref{MainThmConv} of this paper, we could have also assumed $L^{\frac{p}{2}}$, $p > m$ convergence in norm instead of the $L^{\frac{p}{2}}$, $p > m$ bound with volume convergence.

\section{Proof of the Main Theorem}\label{sec:MainThmProof}

In this section we will use $L^{\frac{p}{2}}$ , $p > m$ bounds to prove Theorem \ref{MainThmEst} which will allow us to conclude with the proof of Theorem \ref{MainThmComp}, Theorem \ref{MainThmConv}, and Theorem \ref{TorusTheorem}. In order to control distances between points $q_1,q_2 \in M_j$ from above we will want to build a foliation of a region around the minimizing geodesic connecting $q_1, q_2 \in M_0$.

 \begin{defn}\label{SymFamilyofCurvesDef}
 Let $q_1,q_2 \in M_0$ and $\alpha(t)$ be a length minimizing geodesic joining $q_1$ to $q_2$, with respect to $M_0$, of length $L$. By extending a distance $\tau \in (0,\varepsilon)$ in radial directions orthogonal to $\alpha'(t)$, which can be parameterized over $\Sp^{m-2}$, one obtains a tubular neighborhood $\alpha \subset \bar{U}_{\varepsilon} \subset M$ with coordinates $(\tau,t,\vec{s}) \in [0,\varepsilon]\times [0,L]\times\Sp^{m-2}$. For fixed $(\tau,\vec{s}) \in  [0,\varepsilon] \times \Sp^{m-2}$ we define a curve connecting $q_1$ and $q_2$ in coordinates by
 \begin{align}
 \gamma(\tau,t,\vec{s}) = (\tau L\sin\left(\frac{\pi}{L}t\right),t,\vec{s}).
 \end{align} 
 We define the \textbf{symmetric family of curves joining $q_1$ to $q_2$ of width $\varepsilon$} to be the set $U_{\varepsilon}$ foliated by the curves $\gamma(\tau,t,\vec{s})$ for $(\tau,t,\vec{s}) \in (0,\varepsilon)\times [0,L]\times\Sp^{m-2}$.
 \end{defn} 
 
 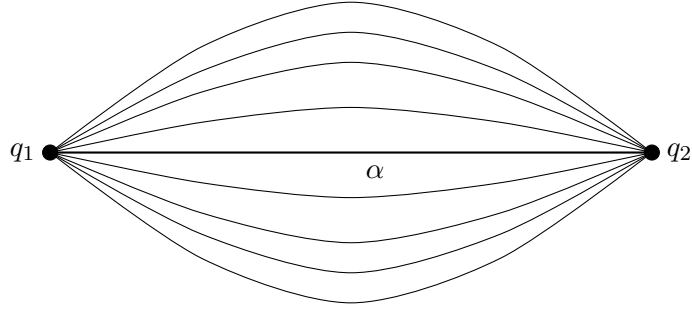
\begin{figure}
 \label{fig:SymmetricFamilyCurves}
 \begin{tikzpicture}[scale=4]
  \draw[thick] (-1,1) -- (1,1);
  \draw[fill] (-1,1) circle [radius=0.025];
  \node[left, outer sep=2pt] at (-1,1) {$q_1$};
  \draw[fill] (1,1) circle [radius=0.025];
  \node[right, outer sep=2pt] at (1,1) {$q_2$};
  \node[right, outer sep=2pt] at (0,1-1/16) {$\alpha$};
   \draw plot [smooth] coordinates {(-1,1)(-.5,1.1)(0,1.15)(.5,1.1)(1,1)};
  \draw plot [smooth] coordinates {(-1,1)(-.5,1-.1)(0,1-.15)(.5,1-.1)(1,1)};
  \draw plot [smooth] coordinates {(-1,1)(-.5,1.2)(0,1.3)(.5,1.2)(1,1)};
  \draw plot [smooth] coordinates {(-1,1)(-.5,1-.2)(0,1-.3)(.5,1-.2)(1,1)};
  \draw plot [smooth] coordinates {(-1,1)(-.5,1.27)(0,1.4)(.5,1.27)(1,1)};
  \draw plot [smooth] coordinates {(-1,1)(-.5,1-.27)(0,1-.4)(.5,1-.27)(1,1)};
  \draw plot [smooth] coordinates {(-1,1)(-.5,1.35)(0,1.5)(.5,1.35)(1,1)};
  \draw plot [smooth] coordinates {(-1,1)(-.5,1-.35)(0,1-.5)(.5,1-.35)(1,1)};
 \end{tikzpicture} 
  \caption{Symmetric family of curves joining $q_1$ to $q_2$ which foliates a region of full volume around $\alpha$. In general, one should note that $\alpha$ is the only geodesic in this family. Since the distance between $q_1$ and $q_2$ is achieved by $\alpha$, which is equal to the infimum of the lengths of the curves in the family, we are able to relate the distance between $q_1$ and $q_2$ to the volume of the foliated region which is the key to the estimate in Theorem \ref{MainThmEst}.}
 \end{figure}
 
 Now that we have this important construction defined we prove some properties about symmetric families of curves which will be used to prove our main H\"{o}lder estimate. We begin with an estimate of the size of the tubular neighborhoood. The author expects that the following result is well known but does not know of a reference in the literature and so has decided to prove the result.
 
 \begin{lem}\label{NorNeighRadiusEst}
 Let $M_0=(M,g_0)$ be a smooth Riemannian manifold defined on the closed manifold $M$. Let $K\in [0,\infty)$ be a bound on the absolute value of the sectional curvature of $M_0$ and $\inj(M_0)$ be the injectivity radius. If  $q_1,q_2\in M_0$ and $\alpha(t)$ a length minimizing geodesic joining $q_1$ to $q_2$, parameterized by arc length, then there exists a $\bar{\varepsilon}(K,\inj(M_0)) > 0$ so that for any $0 < \varepsilon < \bar{\varepsilon}$  a normal cylindrical neighborhood of width $\varepsilon$ exists, $\bar{U}_{\varepsilon}$. 
 \end{lem}
 \begin{proof}
 Our concern is that two normal geodesics emanating from $\alpha$ could intersect and we would like to show that the radial distance at which this can happen is estimated entirely by a bound on the sectional curvature and the injectivity radius. We note that the normal exponential map, $\expon_{\alpha}^{\perp}$, is a diffeomorphism on any neighborhood of zero in the normal bundle to $\alpha$ and hence there exists an $\bar{\varepsilon}> 0$ so that a tubular neighborhood around $\alpha$ exists for all $0 < \varepsilon < \bar{\varepsilon}$. Our goal is to show that $\bar{\varepsilon}$ can be chosen to depend on $K, \inj(M_0)$, and not the points $q_1, q_2$ or $\alpha$. 
 
  Let $\tau'$ be the smallest radial distance so that for $\tau > \tau'$ we have $\expon_{\alpha}^{\perp}$ is no longer a diffeomorphism. If  $\expon_{\alpha}^{\perp}$ is also not a diffeomorphism at $\tau'$ then let $(\tau',t_1,\vec{s}_1), (\tau',t_2,\vec{s}_2) \in (0,\varepsilon)\times [0,L]\times\Sp^{m-2}$ so that $\expon_{\alpha}^{\perp}(\tau',t_1,\vec{s}_1)=\expon_{\alpha}^{\perp}(\tau',t_2,\vec{s}_2)=q_{12} \in M$. If $t_1=t_2$ and $\vec{s}_1\not = \vec{s}_2$ then we can estimate that $\tau' \ge  \inj(M_0)$. If $t_1 \not = t_2$ then let $\beta_{t_1}(\tau)=\expon_{\alpha}^{\perp}(\tau,t_1,\vec{s}_1)$ and $\beta_{t_2}(\tau)=\expon_{\alpha}^{\perp}(\tau,t_2,\vec{s}_2)$. For each $t \in (t_1,t_2)$  we can choose a geodesic $\beta_t(\tau)$ which leaves $\alpha$ orthogonally and meets $\beta_{t_1}$ or $\beta_{t_2}$. If each $\beta_t$ meets $q_{12}$ then $q_{12}$ is a focal point and by Theorem 3.2 of \cite{Warner} we can estimate the occurrence of this focal point in terms of the constant curvature model spaces, as required by the lemma. Instead if there is some $t' \in (t_1,t_2)$ so that $\beta_t$ does not meet $q_{12}$ then this contradicts the assumption that $\tau'$ was the smallest such occurrence of this type. Hence we see that $\bar{\varepsilon}=\tau'$ is estimated in terms of the curvature an injectivity radius in this case.
  
If $\expon_{\alpha}^{\perp}$ is a diffeomorphism at $\tau'$ then consider $\tau_{1,n},\tau_{2,n} > \tau '$ and let $(\tau_{1,n},t_{1,n},\vec{s}_{1,n}), (\tau_{2,n},t_{2,n},\vec{s}_{2,n}) \in (0,\varepsilon)\times [0,L]\times\Sp^{m-2}$ so that $\expon_{\alpha}^{\perp}(\tau_{1,n},t_{1,n},\vec{s}_{1,n})=\expon_{\alpha}^{\perp}(\tau_{2,n},t_{2,n},\vec{s}_{2,n})=q_{n} \in M$ where $\tau_{1,n}\rightarrow \tau'$ and $\tau_{2,n}\rightarrow \tau'$ as $n \rightarrow \infty$. Since $M$, $[0,L]$, and $\Sp^{m-2}$ are compact there must be a subsequence $q_{n_k}$ so that $q_{n_k} \rightarrow q_{\infty}\in M$, $\vec{s}_{i,n_k}\rightarrow \vec{s}_{i,\infty}\in \Sp^{m-2}$, and $t_{i,n_k}\rightarrow t_{i,\infty}\in [0,L]$, $i=1,2$ as $k \rightarrow \infty$ so that $\expon_{\alpha}^{\perp}(\tau_{i,n_k},t_{i,n_k},\vec{s}_{i,n_k})\rightarrow \expon_{\alpha}^{\perp}(\tau',t_{i,\infty},\vec{s}_{i,\infty})=q_{\infty}$, $i=1,2$ as $k \rightarrow \infty$. Let $\beta_{t_{1,n}}(\tau)=\expon_{\alpha}^{\perp}(\tau,t_{1,n},\vec{s}_{1,n})$, $\beta_{t_{2,n}}(\tau)=\expon_{\alpha}^{\perp}(\tau,t_{2,n},\vec{s}_{2,n})$, and $\beta_{t_{i,\infty}}(\tau)=\expon_{\alpha}^{\perp}(\tau',t_{i,\infty},\vec{s}_{i,\infty})$. By construction we have that $\beta_{i,n_k}\rightarrow \beta_{i,\infty}$, $i=1,2$, normal geodesics emanating from $\alpha$ which meet at $q_{\infty}$. Then since the normal exponential map to $\alpha$ is a diffeomorphism at $\tau'$ we must have that $\beta_{t_{1,n}}(\tau)$ and $\beta_{t_{2,n}}(\tau)$ both converge to $\beta_{t_{1,\infty}}(\tau)=\beta_{t_{2,\infty}}(\tau)=:\beta_{\infty}(\tau)$ for $\tau \in [0, \tau']$ as $n \rightarrow \infty$. 
  
  Since $\expon_{\alpha}^{\perp}$ is a diffeomorphism at $\tau'$ but is not a diffeomorphism in a neigborhood of $(\tau',t_{\infty},\vec{s}_{\infty})$  there must exist a $v \in T_{(\tau',t_{\infty},\vec{s}_{\infty})}(T_{\alpha(t_{\infty}))}M)$ so that  $d\expon_{\alpha}^{\perp}(\tau',t_{\infty},\vec{s}_{\infty})[v]=0$. Otherwise, by the inverse function theorem $\expon_{\alpha}^{\perp}$ is a diffeomorphism in a neigborhood of $(\tau',t_{\infty},\vec{s}_{\infty})$. By the Gauss Lemma and the definition of $\expon_{\alpha}^{\perp}$  
 we see that $v \perp \partial_{\tau}$.  Hence we can choose a variation through geodesics normal to $\alpha$ whose variational vector field is a Jacobi field along $\beta_{\infty}(\tau)$ which vanishes at $q_{\infty}$. Hence, $q_{\infty}$ is a focal point and so again by Theorem 3.2 of \cite{Warner} we can estimate the occurrence of this focal point in terms of the constant curvature model spaces, as required by the lemma. Hence we see that $\bar{\varepsilon}=\tau'$ is estimated in terms of the curvature an injectivity radius in this case as well.
  
 \end{proof}

 \begin{lem}\label{MainToolsForProof}
 Let $M_0=(M,g_0)$ be a smooth Riemannian manifold defined on the smooth, connected, closed manifold $M$. Assume  $q_1,q_2\in M_0$ and $\alpha(t)$ a length minimizing geodesic joining $q_1$ to $q_2$, parameterized by arc length. then there exists an $\varepsilon >0$ so that a symmetric family of curves joining $q_1$ to $q_2$ of width $\varepsilon$ exists, $U_{\varepsilon}$. If $K \in [0,\infty)$, is a bound on the absolute value of the sectional curvature, $I_0=\inj(M_0)$ the injectivity radius, and $D_0=\diam(M_0)$ the diameter of $M_0$ then for $\bar{\varepsilon}=\bar{\varepsilon}(K, \inj(M_0))$ from Lemma \ref{NorNeighRadiusEst} we have that for $0<\varepsilon \le \frac{\bar{\varepsilon}}{\diam(M_0)} $ a symmetric family of curves joining $q_1$ to $q_2$ of width $\varepsilon$ exists, $U_{\varepsilon}$,  and with $\gamma'=\frac{d\gamma}{dt}$ we find
 \begin{align}
 |\gamma'(\tau,t,\vec{s})|_{g_0} \le C(K,I_0,D_0), \quad (\tau,t,\vec{s}) \in [0,\varepsilon]\times [0,L]\times \Sp^{m-2}.
\end{align}  
For $\pi:[0,\varepsilon]\times [0,L]\times \Sp^{m-2} \rightarrow [0,\varepsilon]\times \Sp^{m-2}$ given by $\pi(\tau,t,\vec{s})=(\tau,\vec{s})$ and $\bar{\gamma}^{-1}=\pi \circ\gamma^{-1} $ we find the normal jacobian of $\bar{\gamma}^{-1}$ to be
\begin{align}
NJ\bar{\gamma}^{-1}&=\left(L\sin\left(\frac{\pi}{L}t\right) \right)^{1-m}\sqrt{\det_{[\gamma']_{g}^{\perp}}(T)}.
\end{align}
for some map $T:U_{\varepsilon} \rightarrow [0,\varepsilon]\times \Sp^{m-2}$ where 
\begin{align}
\sqrt{\det_{[\gamma']_{g_0}^{\perp}}(T)}
&\le C(K,m, I_0, D_0).
\end{align}
 \end{lem}
 \begin{proof}
 Let $\alpha(t)$ be a length minimizing geodesic joining $q_1, q_2 \in M$ of length $L$, parameterized by arc length, and let $U_{\varepsilon}$, $0 < \varepsilon \le \frac{\bar{\varepsilon}}{\diam(M_0)},$ be a symmetric family of curves joining $q_1$ and $q_2$, guaranteed to exist by Lemma \ref{NorNeighRadiusEst} and the fact that $\tau L\sin\left(\frac{\pi}{L}t\right) \le \tau L \le\frac{\bar{\varepsilon}}{\diam(M_0)} \diam(M_0)\le \bar{\varepsilon}$. $U_{\varepsilon}$ is foliated by the curves given in coordinates by
 \begin{align}
 \gamma(\tau,t,\vec{s}) = (\tau L\sin\left(\frac{\pi}{L}t\right),t,\vec{s}),
 \end{align}
 and hence
  \begin{align}
 \gamma'(\tau,t,\vec{s}) = (\tau \pi\cos\left(\frac{\pi}{L}t\right),1,0).
 \end{align}
 On the neighborhood of $\alpha$ given by $\bar{U}_{\varepsilon}$ the metric $g_0$ can be written in the coordinates $(\tau,t,\vec{s}) \in [0,\varepsilon]\times [0,L]\times \Sp^{m-2}$ as
 \begin{align}
 g_0=d\tau^2 +\bar{g}_{\tau},
 \end{align}
 where $\bar{g}_{\tau}$ is non-zero for vectors tangent to $[0,L]\times \Sp^{m-2}$ and zero otherwise.

 Now we write a constant curvature $K\in \R$ metric $g_K$ in terms of the coordinates on the neighborhood of $\alpha$ as 
 \begin{align}
\hat{g}_K= d\tau^2+\lambda_K(\tau)^2dt^2+\eta_K(\tau)^2\sigma_{m-2},
 \end{align}
 where $\sigma_{m-2}$ is the standard round metric on a $m-2$ dimensional sphere,
 \begin{align}
 \lambda_K(\tau)=
 \begin{cases}
 \cos(\sqrt{K}\tau)& K > 0
 \\ 1 & K=0
 \\ \cosh(\sqrt{|K|} \tau)& K < 0
 \end{cases}
 \end{align}
 and
  \begin{align}
  \eta_K(\tau)=
 \begin{cases}
  \frac{1}{\sqrt{K}}\sin(\sqrt{K}\tau)& K > 0
 \\ \tau & K=0
 \\ \frac{1}{\sqrt{|K|}}\sinh(\sqrt{|K|} \tau)& K < 0
 \end{cases}.
 \end{align}
 Now notice that $\partial_t=\partial_0$ is a normal Jacobi field for $g_0$ and there exists an orthonormal set $\partial_1,...,\partial_{m-2}\in T_{\vec{s}}\Sp^{m-2}$, with respect to $\sigma$, of normal Jacobi fields tangent to $\Sp^{m-2}$. Now notice that
 \begin{align}
 g_0(\partial_t,\partial_t)|_{\tau=0}&=1,\label{JacobiEq1}
\\ g_0(\nabla_{\partial_{\tau}}\partial_t,\nabla_{\partial_{\tau}}\partial_t)|_{\tau=0}&=0,\label{JacobiEq2}
 \\ g_0(\partial_i,\partial_i)|_{\tau=0}&=0, \quad 1 \le i \le m-2,\label{JacobiEq3}
 \\ g_0(\nabla_{\partial_{\tau}}\partial_i,\nabla_{\partial_{\tau}}\partial_i)|_{\tau=0}&=1, \quad 1 \le i \le m-2,\label{JacobiEq4}
 \end{align}
 where \eqref{JacobiEq1} follows since $\alpha(t)$ is a geodesic which is parameterized by arc length  and \eqref{JacobiEq3}, \eqref{JacobiEq4} follow from standard calculations for Riemannian metrics in polar coordinates. For \eqref{JacobiEq2} we calculate
 \begin{align}
  g_0(\nabla_{\partial \tau}\partial_t,\partial_t)|_{\tau=0}&= g_0(\nabla_{\partial t}\partial_{\tau},\partial_t)|_{\tau=0}
  \\&=  \partial_t(g_0(\partial_{\tau},\partial_t))|_{\tau=0}-g_0(\partial_t,\nabla_{\partial_t}\partial_t)|_{\tau=0}=0,
 \\0=\partial_{\tau}(g_0(\partial_t,\partial_{\tau}))|_{\tau=0}&= g_0(\nabla_{\partial \tau}\partial_t,\partial_{\tau})|_{\tau=0}+g_0(\partial_t,\nabla_{\partial \tau}\partial_{\tau})|_{\tau=0}
 \\&= g_0(\nabla_{\partial \tau}\partial_t,\partial_{\tau})|_{\tau=0},
 \end{align}
where we note that when $\tau=0$ the vector field $\partial_{\tau}$ spans the normal space to $\alpha(t)$. 
 
 Hence we can apply the extension of the Rauch Comparison Theorem  given by Berger \cite{Berger} and Warner \cite{Warner} with $K,k \in \R$, $K \le k$ upper and lower bounds on the sectional curvature of $M_0$, to conclude 
 \begin{align}
\lambda_k(\tau)^2 \le  g_0(\partial_t,\partial_t) &\le \lambda_K(\tau)^2 = \hat{g}_{K}(\partial_t,\partial_t)\label{LengthEst1}
  \\\eta_k(\tau)^2\le g_0(\partial_i,\partial_i)& \le \eta_K(\tau)^2 =\hat{g}_{K}(\partial_i,\partial_i), \quad 1 \le i \le m-2.\label{LengthEst2}
 \end{align} 
 
By the comparison theorem of Heintze and Karcher \cite{HeintzeKarcher} we find
\begin{align}
\lambda_k(\tau)^{m-2}\eta_k(\tau)=|\partial_{1}\wedge ... \wedge \partial_{m-2}\wedge \partial_t|_{\hat{g}_k} \le |\partial_{1}\wedge ... \wedge \partial_{m-2}\wedge \partial_t|_{g_0},
\end{align}
and now by applying the Hadamard-Schwarz inequality \cite{IKK} we find (a hat over a vector implies it is missing from the calculation)
\begin{align}
|\partial_{1}\wedge ... \wedge \partial_{m-2}\wedge \partial_t|_{g_0} &\le C_m |\partial_{1}\wedge ... \wedge \hat{\partial}_{i} \wedge... \wedge \partial_{m-2}|_{g_0}| \partial_{i}\wedge \partial_t|_{g_0}
\\&\le C_m |\partial_{1}|_{g_0} ... | \hat{\partial}_{i}|_{g_0} ... | \partial_{m-2}|_{g_0}| \partial_{i}\wedge \partial_t|_{g_0}
\\&\le C_m\lambda_{K}(\tau)^{m-3} \sqrt{g_0(\partial_t,\partial_t)g_0(\partial_i,\partial_i)-g_0(\partial_t,\partial_i)^2},
\end{align}
which implies for $\varepsilon < \bar{\varepsilon}$ the lower bound
\begin{align}
\sqrt{g_0(\partial_t,\partial_t)g_0(\partial_i,\partial_i)-g_0(\partial_t,\partial_i)^2 }\ge \frac{\lambda_k(\tau)^{m-2}\eta_k(\tau)}{C_m\lambda_{K}(\tau)^{m-3}} \ge C(K,k,m)\tau \label{AngleEst}
\end{align}

Additionally, we can estimate for $\varepsilon < \bar{\varepsilon}$ 
 \begin{align}
 |\gamma'(\vec{s},\tau,t)|_{g_0} &= \sqrt{g_0(\partial_{\tau},\partial_{\tau})d\tau(\gamma')^2+g_0(\partial_t,\partial_t)dt(\gamma')^2}
 \\&\le \sqrt{ d\tau(\gamma')^2+\lambda_K(\tau)^2dt(\gamma')^2}
 \\&=\sqrt{\tau^2 \pi^2 \cos^2\left(\frac{\pi}{L}t\right)+ \cosh^2(\sqrt{|K|}\tau)}
 \\&\le \sqrt{\varepsilon^2 \pi^2 + \cosh^2(\sqrt{|K|}\varepsilon)} \le C(k,K, I_0,D_0),
 \end{align}
 where we use the uniform upper bound on $\varepsilon$ for the last inequality.

Furthermore, if we define the set $X_{\varepsilon}=[0,\varepsilon]\times [0,L]\times\Sp^{m-2}$ and look at the map
\begin{align}
\gamma&: (X_{\varepsilon}, \hat{g}_0)\rightarrow (U_{\varepsilon},g_0),\quad \gamma(\tau,t,\vec{s}) = (\tau L\sin\left(\frac{\pi}{L}t\right),t,\vec{s}),
\end{align}
then we find

\begin{align}
d\gamma=
\begin{bmatrix}
L\sin\left(\frac{\pi}{L}t\right) & \tau \pi \cos\left(\frac{\pi}{L}t\right) & 0 \\
0 & 1 & 0 \\
0 & 0 &  L\sin\left(\frac{\pi}{L}t\right)I
\end{bmatrix},
\end{align}
as well as,
\begin{align}
d\gamma^{-1}=
\begin{bmatrix}
\frac{1}{L\sin\left(\frac{\pi}{L}t\right)} & -\frac{\tau \pi \cos\left(\frac{\pi}{L}t\right)}{L\sin\left(\frac{\pi}{L}t\right)} & 0 \\
0 & 1 & 0 \\
0 & 0 & \frac{1}{ L\sin\left(\frac{\pi}{L}t\right)}I
\end{bmatrix}.
\end{align}
Now if we let $W_{\varepsilon} = [0,\varepsilon]\times\Sp^{m-2}$ and define
\begin{align}
\bar{\gamma}^{-1}&:  (U_{\varepsilon},g_0=d\tau^2+\bar{g}_{\tau})\rightarrow(W_{\varepsilon},\hat{g}_0|_{W_{\varepsilon}}=d\tau^2+\tau^2\sigma),
\end{align}
then we know that 
\begin{align}
d\bar{\gamma}^{-1}=
\begin{bmatrix}
\frac{1}{L\sin\left(\frac{\pi}{L}t\right)} & -\frac{\tau \pi \cos\left(\frac{\pi}{L}t\right)}{L\sin\left(\frac{\pi}{L}t\right)} & 0 \\
0 & 0 & \frac{1}{ L\sin\left(\frac{\pi}{L}t\right)}I
\end{bmatrix}.
\end{align}

We are in particular interested in the subspace which is $g_0$ orthogonal to $\gamma'$, $W=[\gamma']_{g_0}^{\perp}$, which if $e_1=(\tau \pi \cos\left(\frac{\pi}{L}t\right),1,0,...,0)$ then the space is spanned by
\begin{align}
e_2&= (g_0(\partial_t,\partial_t), -\tau \pi \cos\left(\frac{\pi}{L}t\right),0,...,0) 
\\e_i&= (0,-g_0(\partial_t,\partial_i),0,...,g_0(\partial_t,\partial_t),0,...,0), \quad 3 \le i \le m,
\end{align}
where if we define $\bar{e}_i = \frac{e_i}{\|e_i\|_{g_0}}$, $1 \le i \le m$ then $\{\bar{e}_1,...,\bar{e}_m\}$ is a $g_0$ orthonormal basis where 
\begin{align}
\bar{e}_1&=\frac{(\tau \pi \cos\left(\frac{\pi}{L}t\right),1,0,...,0)}{\sqrt{\tau^2 \pi^2\cos^2\left(\frac{\pi}{L}t\right)+g_0(\partial_t,\partial_t)}} 
\\\bar{e}_2&= \frac{(g_0(\partial_t,\partial_t), -\tau \pi \cos\left(\frac{\pi}{L}t\right),0,...,0)}{\sqrt{g_0(\partial_t,\partial_t)}\sqrt{\tau^2 \pi^2\cos^2\left(\frac{\pi}{L}t\right)+g_0(\partial_t,\partial_t)}}
\\\bar{e}_i&= \frac{(0,-g_0(\partial_t,\partial_i),0,...,g_0(\partial_t,\partial_t),0,...,0)}{\sqrt{g_0(\partial_t,\partial_t)(g_0(\partial_t,\partial_t)g_0(\partial_i,\partial_i)-g_0(\partial_t,\partial_i)^2)}}, \quad 3 \le i \le m.
\end{align} 

We can also define $\bar{f}_i=d\bar{\gamma}^{-1}(\bar{e}_i)$ for $2 \le i \le m$ where
\begin{align} 
\bar{f}_2&= \frac{(g_0(\partial_t,\partial_t)+ \tau^2 \pi^2 \cos^2\left(\frac{\pi}{L}t\right),0,...,0)}{L\sin\left(\frac{\pi}{L}t\right)\sqrt{g_0(\partial_t,\partial_t)}\sqrt{\tau^2 \pi^2\cos^2\left(\frac{\pi}{L}t\right)+g_0(\partial_t,\partial_t)}}
\\\bar{f}_i&= \frac{(\tau \pi\cos\left(\frac{\pi}{L}t\right) g_0(\partial_t,\partial_i),0,...,g_0(\partial_t,\partial_t),0,...,0)}{ L\sin\left(\frac{\pi}{L}t\right)\sqrt{g_0(\partial_t,\partial_t)(g_0(\partial_t,\partial_t)g_0(\partial_i,\partial_i)-g_0(\partial_t,\partial_i)^2)}}, 
\\& \qquad 3 \le i \le m.
\end{align}
So we can calculate the inner products where $3 \le i,j \le m$
\begin{align}
\hat{g}&_0|_{W_{\varepsilon}}(\bar{f}_i,\bar{f}_j)
\\&=\frac{\tau^2( \pi^2\cos^2\left(\frac{\pi}{L}t\right) g_0(\partial_t,\partial_i)g_0(\partial_t,\partial_j)+g_0(\partial_t,\partial_t)^2)}{L^2\sin^2\left(\frac{\pi}{L}t\right)g_0(\partial_t,\partial_t)\displaystyle\prod_{k=i,j}\sqrt{g_0(\partial_t,\partial_t)g_0(\partial_k,\partial_k)-g_0(\partial_t,\partial_k)^2}},
\\\hat{g}&_0|_{W_{\varepsilon}}(\bar{f}_2,\bar{f}_i)
\\&= \frac{\tau\sqrt{g_0(\partial_t,\partial_t)+ \tau^2 \pi^2 \cos^2\left(\frac{\pi}{L}t\right)} \pi\cos\left(\frac{\pi}{L}t\right) g_0(\partial_t,\partial_i)}{L^2\sin^2\left(\frac{\pi}{L}t\right)g_0(\partial_t,\partial_t)\sqrt{g_0(\partial_t,\partial_t)g_0(\partial_i,\partial_i)-g_0(\partial_t,\partial_i)^2}},
\\\hat{g}&_0|_{W_{\varepsilon}}(\bar{f}_2,\bar{f}_2)
\\&=\frac{\left(g_0(\partial_t,\partial_t)+ \tau^2 \pi^2 \cos^2\left(\frac{\pi}{L}t\right)\right)^2}{L^2\sin^2\left(\frac{\pi}{L}t\right)g_0(\partial_t,\partial_t)\left(\tau^2 \pi^2\cos^2\left(\frac{\pi}{L}t\right)+g_0(\partial_t,\partial_t)\right)}
\\&=\frac{g_0(\partial_t,\partial_t)+ \tau^2 \pi^2 \cos^2\left(\frac{\pi}{L}t\right)}{g_0(\partial_t,\partial_t)L^2\sin^2\left(\frac{\pi}{L}t\right)}.
\end{align}
By \eqref{LengthEst1}, \eqref{LengthEst2}, and \eqref{AngleEst} we see that for $ 3 \le i,j \le m,$
\begin{align}
\hat{g}_0|_{W_{\varepsilon}}(L\sin\left(\frac{\pi}{L}t\right)\bar{f}_{2},L\sin\left(\frac{\pi}{L}t\right)\bar{f}_{2}) &\le C(k,K,m,I_0,D_0),
\\\hat{g}_0|_{W_{\varepsilon}}(L\sin\left(\frac{\pi}{L}t\right)\bar{f}_{2},L\sin\left(\frac{\pi}{L}t\right)\bar{f}_{i}) &\le C(k,K,m,I_0,D_0),
\\\hat{g}_0|_{W_{\varepsilon}}(L\sin\left(\frac{\pi}{L}t\right)\bar{f}_{i},L\sin\left(\frac{\pi}{L}t\right)\bar{f}_{j}) &\le C(k,K,m,I_0,D_0),
\end{align}

Then we can estimate the normal jacobian of $\bar{\gamma}^{-1}$ to be
\begin{align}
NJ\bar{\gamma}^{-1}&=\sqrt{\det\left(\hat{g}_0|_{W_{\varepsilon}}(\bar{f}_{\alpha},\bar{f}_{\beta})_{2 \le \alpha,\beta \le m}\right)}
\\&=\left(L\sin\left(\frac{\pi}{L}t\right) \right)^{1-m}
\\&\quad \cdot \sqrt{\det\left(\hat{g}_0|_{W_{\varepsilon}}\left(L\sin\left(\frac{\pi}{L}t\right)\bar{f}_{\alpha},L\sin\left(\frac{\pi}{L}t\right)\bar{f}_{\beta}\right)_{2 \le \alpha,\beta \le m}\right)}
\\&= \left(L\sin\left(\frac{\pi}{L}t\right) \right)^{1-m}\sqrt{\det_{[\gamma']_{g_0}^{\perp}}(T)}
\\&\le \left(L\sin\left(\frac{\pi}{L}t\right) \right)^{1-m}C(k,K,m,I_0,D_0),
\end{align}
where we define
\begin{align}
\sqrt{\det_{[\gamma']_{g_0}^{\perp}}(T)}=\sqrt{\det\left(\hat{g}_0|_{W_{\varepsilon}}\left(L\sin\left(\frac{\pi}{L}t\right)\bar{f}_{\alpha},L\sin\left(\frac{\pi}{L}t\right)\bar{f}_{\beta}\right)_{2 \le \alpha, \beta\le m}\right)}.
\end{align}

 \end{proof}

 We now use the symmetric family of curves combined with Lemma \ref{MainToolsForProof} to obtain H\"{o}lder control on distances from above.

\begin{thm}\label{ConfDistBoundAbove}
Let $g_0$ be a smooth Riemannian metric and $g_1$ a continuous Riemannian metric defined on the smooth, connected, closed  manifold $M$. If
\begin{align}
\|g_1\|_{L_{g_0}^{\frac{p}{2}}(M)} \le C,\quad p > m,
\end{align}
then we can bound the distance between $q_1,q_2$, with respect to $g_1$, by
\begin{align}
d_{g_1}(q_1,q_2) \le C'(K,m,p,I_0,D_0,C) d_{g_0}(q_1,q_2)^{\frac{p-m}{p}} 
\end{align}
where $K\in [0,\infty)$ is a bound on the absolute value of the sectional curvature of $M_0$, $\inj(M,g_0)$ is the injectivity radius of $M_0$, and $D_0=\diam(M_0)$ is the diameter of $M_0$.
\end{thm}
\begin{proof}
Let $U_{\varepsilon}$ be a symmetric family of curves joining $q_1$ to $q_2$ of width $\varepsilon$ with coordinates $(\tau,t,\vec{s}) \in (0,\varepsilon)\times [0,L]\times\Sp^{m-2}$ where the metric can be written as $g_0=d\tau^2+\bar{g}_0$ and let $W_{\varepsilon} = [0,\varepsilon]\times\Sp^{m-2}$. 

By using the minimizing properties of geodesics we find
\begin{align}
\vol(W_{\varepsilon}) d_{g_1}(p,q) &\le \vol(W_{\varepsilon}) \lp\min_{(\vec{s},\tau)\in W_{\varepsilon}} L_{g_1}(\gamma(\vec{s},\tau,\cdot))\rp 
\\&\le \int_{W_{\epsilon}} L_{g_1}(\gamma(\vec{s},\tau,\cdot))\tau^{m-2}dV_{\sigma_{m-2}}d\tau, \label{WEpsilonStart}
\end{align}
where $\vol(W_{\varepsilon})=\omega_{m-2}\varepsilon^{m-1}$ is the volume in polar coordinates for Euclidean space where $\omega_{m-2}$ is the volume of the standard unit $m-2$ dimensional sphere.
By rewriting the length, using Lemma \ref{NormComparisonLemma}, and then using the coarea formula in \eqref{WEpsilonStart} we find
\begin{align}
\vol(W_{\varepsilon})& d_{g_1}(q_1,q_2)
\\&\le \int_{W_{\varepsilon}} \int_0^L \sqrt{g_1(\gamma',\gamma')} dt \tau^{m-2}dV_{\sigma_{m-2}}d\tau
\\&\le \int_{W_{\varepsilon}} \int_0^L |g_1|_{g_0}^{\frac{1}{2}}|\gamma'|_{g_0} dt \tau^{m-2}dV_{\sigma_{m-2}}d\tau
\\&=  L^{1-m} \int_{U_{\varepsilon}} |g_1|_{g_0}^{\frac{1}{2}}\lp\sin\lp\frac{\pi}{L}t\rp\rp^{1-m}\sqrt{\det_{[\gamma']_{g_0}^{\perp}}(T ) }dV_{g_0}, \label{WEpsilonMiddle}
\end{align}
where $\det_{[\gamma']_{g_0}^{\perp}}(T)$ is defined in Lemma \ref{MainToolsForProof}. 

By moving $\vol(W_{\varepsilon})$ to the other side and using Holder's inequality with $p>1$, $q = \frac{p}{p-1}$
\begin{align}
 d_{g_1}&(q_1,q_2)
 \\&\le \vol(W_{\varepsilon})^{-1}L^{1-m}  \lp\int_{U_{\varepsilon}} |g_1|_{g_0}^{\frac{p}{2}}dV_{g_0}\rp^{1/p}
\\&\cdot 
 \lp \int_{U_{\varepsilon}}\sin\lp\frac{\pi}{L}t \rp^{-\frac{p(m-1)}{p-1}} \left(\sqrt{\det_{[\gamma']_{g_0}^{\perp}}(T )}\right)^{\frac{p}{p-1}}dV_{g_0} \rp^{\frac{p-1}{p}}.\label{WEpsilonEnd}
\end{align}

By Lemma \ref{MainToolsForProof}
\begin{align}
|\gamma'(\vec{s},s,t)|_{g_0} \le C(K,I_0, D_0)\text{ for all } (\vec{s},s,t)\in \Sp^{m-2}\times (0,\epsilon)\times[0,L].\label{gammaBound}
\end{align}

By using the coarea formula and \eqref{gammaBound} we find
\begin{align}
&\lp\int_{U_{\varepsilon}}\sin\lp\frac{\pi}{L}t\rp^{-\frac{p(m-1)}{p-1}} \left(\sqrt{\det_{[\gamma']_{g_0}^{\perp}}(T )}\right)^{\frac{p}{p-1}}dV_{g_0}\rp ^{\frac{p-1}{p}}
\\& = \lp \int_0^L\int_0^{\varepsilon} \int_{\Sp^{m-2}}I\tau^{m-2}\left(\sqrt{\det_{[\gamma']_{g_0}^{\perp}}(T )}\right)^{\frac{1}{p-1}}dV_{\sigma_{m-2}} d \tau |\gamma'|_{g_0}dt\rp^{\frac{p-1}{p}},
\\ I&=L^{m-1} \sin\lp\frac{\pi}{L}t\rp^{-\frac{p(m-1)}{p-1}+m-1},
\\&\le \bar{C}L^{\frac{(m-1)(p-1)}{p}} \lp\int_0^L\sin\lp\frac{\pi}{L}t\rp^{-\frac{p(m-1)}{p-1}+m-1}  dt\rp^{\frac{p-1}{p}} \label{TrigBound},
\end{align}
where
\begin{align}
\bar{C}(K,m,p, I_0, D_0)
=\lp  \frac{\bar{\varepsilon}^{m-1}}{m-1} \omega_{m-2} C^2 \rp^{\frac{p-1}{p}}.
\end{align}

Now if we reconcile \eqref{WEpsilonEnd} and \eqref{TrigBound} we find
\begin{align}
d_{g_1}(q_1,q_2) &\le   \vol(W_{\epsilon})^{-1}\bar{C}(K,m,p, I_0,D_0) L^{\frac{(m-1)(p-1)}{p}+1-m}
\\&\quad \cdot \lp\int_{U_{\varepsilon}} |g_1|_{g_0}^{\frac{p}{2}} dV_{g_0}\rp^{1/p} 
\lp\int_0^L\sin\lp\frac{\pi}{L}t\rp^{-\frac{m-1}{p-1}}  dt\rp^{\frac{p-1}{p}}\label{ReconciledEq}
\end{align}

So if we choose $p > m$ then 
\begin{align}
-\frac{p(m-1)}{p-1}+m-1  = \frac{(m-1)(p-1) - p(m-1)}{p-1} = -\frac{m-1}{p-1} > -1
\end{align} 
and hence the integral of the power of $\sin$ is integrable,
\begin{align}
&\lp\int_0^L \sin\lp\frac{\pi}{L}t\rp ^{-\frac{p(m-1)}{p-1}+m-1}  dt\rp^{\frac{p-1}{p}} 
\\& =L^{\frac{p-1}{p}}\lp\int_0^1 \sin\lp\pi s\rp ^{-\frac{p(m-1)}{p-1}+m-1}  ds\rp^{\frac{p-1}{p}}\le L^{\frac{p-1}{p}}\tilde{C}_{p}.\label{TrigIntegrable}
\end{align}

Furthermore,  we know that
\begin{align}
&\frac{p-1}{p}+\frac{(m-1)(p-1)}{p}+1-m
\\&= \frac{p-1+(m-1)(p-1)+(1-m)p}{p}=\frac{p-m}{p},\label{FinalPower}
\end{align}
and if we choose $\varepsilon' = \frac{\bar{\varepsilon}}{\diam(M_0)} $  then since $\bar{\varepsilon}$ cannot be arbitrarily small by Lemma \ref{NorNeighRadiusEst} we find
\begin{align}\label{WVolumeEst}
\vol(W_{\epsilon'})^{-1}= \omega_{m-2}^{-1} \bar{\varepsilon}^{1-m} \diam(M_0)^{m-1}   \le C(K,I_0,D_0).
\end{align}
Putting \eqref{ReconciledEq}, \eqref{TrigIntegrable}, \eqref{FinalPower}, and \eqref{WVolumeEst} together we find 
\begin{align}
d_{g_1}(q_1,q_2) \le L^{\frac{p-m}{p}}\tilde{C}(K,m,p,I_0,D_0) \lp\int_{U_{\varepsilon'}} |g_1|_{g_0}^{\frac{p}{2}}dV_{g_0}\rp^{1/p}.\label{OneFromTheEnd}
\end{align}

Then by substituting $L = d_{g_0}(q_1,q_2)$ in \eqref{OneFromTheEnd} and using the fact that
\begin{align}
\lp\int_{U_{\varepsilon}} |g_1|_{g_0}^{\frac{p}{2}}dV_{g_0}\rp^{1/p} \le \lp\int_{M} |g_1|_{g_0}^{\frac{p}{2}}dV_{g_0}\rp^{1/p}
\end{align}
we find the desired bound on distance.
\end{proof}

We now finish with the proofs of the other two main theorems.

\begin{proof}[Proof of Theorem \ref{MainThmComp}]
By Theorem \ref{MainThmEst} we know that
\begin{align}
 d_j(q_1,q_2) \le C'(K,m,p,I_0,D_0,C) d_0(q_1,q_2)^{\frac{p-m}{p}},
\end{align}
and hence by the Arzella-Ascolli theorem we know there exists a continuos function $d_{\infty}:M \times M \rightarrow [0,\infty)$ so that $d_j \rightarrow d_{\infty}$ uniformly as functions. Now if we define $M_{\infty}=(M/d_{\infty},d_{\infty})$ then by the consequences of Corollary 7.3.28 of \cite{BBI} we know that $M_j \GHto M_{\infty}$.

In the case where we have assumed a H\"{o}lder lower bound we know that
\begin{align}
cd_0(q_1,q_2)^{\eta} \le d_j(q_1,q_2) \le C'(K,m,p,I_0,D_0,C) d_0(q_1,q_2)^{\frac{p-1}{p}},
\end{align}
and hence $M/d_{\infty}=M$.
\end{proof}

\begin{proof}[Proof of Theorem \ref{MainThmConv}]
By Theorem \ref{MainThmComp} we know that there is a limiting metric space $d_{\infty}$ so that $M_j \GHto M_{\infty}$ where $M_{\infty}=(M,d_{\infty})$. By Theorem \ref{VADB} we know that $M_j \VFto M_0$. By  Theorem 3.20 of \cite{SW-JDG}, stated exactly as required in Theorem 2.30 of \cite{Sormani-ArzAsc},  we know that $M_0 \subset M_{\infty}$ and hence $M_0=M_{\infty}$, as desired. The measured Gromov-Hausdorff convergence follows from the volume preserving intrinsic flat convergence by Theorem 2.4 of \cite{Allen-Sormani-2}. 
\end{proof}

\begin{proof}[Proof of Theorem \ref{TorusTheorem}]
Since each $g_{0,j}$ has either constant negative or zero scalar curvature we can pick a subsequence which is either entirely zero scalar curvature or entirely negative scalar curvature which meets the hypotheses of Theorem 1.2 or Theorem 1.4 of \cite{Allen-Tori}. Also, notice that Theorem \ref{MainThmConv} implies a diameter bound. Then by Lemma 4.4 of \cite{Allen-Tori} we see that the $L^{\frac{p}{2}}, p>m$ bound implies the necessary uniform integrability condition to apply Theorem 1.2 or Theorem 1.4 of \cite{Allen-Tori} to conclude volume preserving Sormani-Wenger intrinsic flat convergence to a flat torus $\bar{\Tor}_0^m$. Then by Theorem \ref{MainThmComp} we know that there is a limiting metric space $d_{\infty}$ so that $M_j \GHto M_{\infty}$ where $M_{\infty}=(\mathbb{T}^m/d_{\infty},d_{\infty})$. By  Theorem 3.20 of \cite{SW-JDG}, stated exactly as required in Theorem 2.30 of \cite{Sormani-ArzAsc},  we know that $\bar{\Tor}_0^m \subset \mathbb{T}^m/d_{\infty}$ and hence $\bar{\Tor}_0^m=\mathbb{T}^m/d_{\infty}$, as desired. The measured Gromov-Hausdorff convergence follows from the volume preserving intrinsic flat convergence by Theorem 2.4 of \cite{Allen-Sormani-2}. 
\end{proof}

  \bibliographystyle{alpha}
 \bibliography{Allen}

\newcommand{\etalchar}[1]{$^{#1}$}
\begin{thebibliography}{AHP{\etalchar{+}}18}

\bibitem[AB19]{Allen-Bryden}
Brian Allen and Edward Bryden.
\newblock Sobolev bounds and convergence of {R}iemannian manifolds.
\newblock {\em Nonlinear Anal.}, 185:142--169, 2019.

\bibitem[AC91]{Anderson-Cheeger}
Michael~T. Anderson and Jeff Cheeger.
\newblock Diffeomorphism finiteness for manifolds with {R}icci curvature and
  {$L^{n/2}$}-norm of curvature bounded.
\newblock {\em Geom. Funct. Anal.}, 1(3):231--252, 1991.

\bibitem[ACTar]{Aldana-Carron-Tapie}
Clara~L. Aldana, Gilles Carron, and Samuel Tapie.
\newblock $a_{\infty}$ weights and compactness of conformal metrics under
  $l^{n/2}$ curvature bounds.
\newblock {\em Analysis and PDE}, To Appear.

\bibitem[AHP{\etalchar{+}}18]{AHPPW}
B.~Allen, L.~Hernandez, D.~Parise, A.~Payne, and S.~Wang.
\newblock Warped tori with almost non-negative scalar curvature.
\newblock {\em Geometriae Dedicata}, 200(2), 2018.

\bibitem[All21]{Allen-Tori}
Brian Allen.
\newblock Almost non-negative scalar curvature on riemannian manifolds
  conformal to tori.
\newblock {\em Journal of Geometric Analysis}, 31:11190--11213, 2021.

\bibitem[And90]{Anderson-Ricci}
Michael~T. Anderson.
\newblock Convergence and rigidity of manifolds under {R}icci curvature bounds.
\newblock {\em Invent. Math.}, 102(2):429--445, 1990.

\bibitem[And05]{Anderson-Orbifold}
Michael~T. Anderson.
\newblock Orbifold compactness for spaces of {R}iemannian metrics and
  applications.
\newblock {\em Math. Ann.}, 331(4):739--778, 2005.

\bibitem[APS20]{Allen-Perales-Sormani}
Brian Allen, Raquel Perales, and Christina Sormani.
\newblock Volume above distance below.
\newblock {\em arXiv:2003.01172 [math.MG]}, 2020.

\bibitem[AS19]{Allen-Sormani}
Brian Allen and Christina Sormani.
\newblock Contrasting various notions of convergence in geometric analysis.
\newblock {\em Pacific Journal of Mathematics}, 303(1):1--46, 2019.

\bibitem[AS20]{Allen-Sormani-2}
Brian Allen and Christina Sormani.
\newblock Relating notions of convergence in geometric analysis.
\newblock {\em Nonlinear Analysis}, 200, 2020.

\bibitem[BBI01]{BBI}
Dmitri Burago, Yuri Burago, and Sergei Ivanov.
\newblock {\em A course in metric geometry}, volume~33 of {\em Graduate Studies
  in Mathematics}.
\newblock American Mathematical Society, Providence, RI, 2001.

\bibitem[Ber62]{Berger}
M.~Berger.
\newblock An extension of rauch's metric comparison theorem and some
  applications.
\newblock {\em Illinois Journal of Mathematics}, 6:700--712, 1962.

\bibitem[CC97]{Cheeger-Colding-1}
Jeff Cheeger and Tobias~H. Colding.
\newblock On the structure of spaces with {R}icci curvature bounded below. {I}.
\newblock {\em J. Differential Geom.}, 46(3):406--480, 1997.

\bibitem[Col96]{Colding-shape}
Tobias~H. Colding.
\newblock Shape of manifolds with positive {R}icci curvature.
\newblock {\em Invent. Math.}, 124(1-3):175--191, 1996.

\bibitem[Col97]{Colding-volume}
Tobias~H. Colding.
\newblock Ricci curvature and volume convergence.
\newblock {\em Ann. of Math. (2)}, 145(3):477--501, 1997.

\bibitem[Gao90]{Gao-integral1}
L.~Zhiyong Gao.
\newblock Convergence of {R}iemannian manifolds; {R}icci and
  {$L^{n/2}$}-curvature pinching.
\newblock {\em J. Differential Geom.}, 32(2):349--381, 1990.

\bibitem[Gro81a]{Gromov-poly}
Mikhael Gromov.
\newblock Groups of polynomial growth and expanding maps.
\newblock {\em Inst. Hautes \'Etudes Sci. Publ. Math.}, (53):53--73, 1981.

\bibitem[Gro81b]{Gromov-metric}
Mikhael Gromov.
\newblock {\em Structures m\'etriques pour les vari\'et\'es riemanniennes},
  volume~1 of {\em Textes Math\'ematiques [Mathematical Texts]}.
\newblock CEDIC, Paris, 1981.
\newblock Edited by J. Lafontaine and P. Pansu.

\bibitem[Gro14]{GroD}
Misha Gromov.
\newblock Dirac and {P}lateau billiards in domains with corners.
\newblock {\em Cent. Eur. J. Math.}, 12(8):1109--1156, 2014.

\bibitem[HK78]{HeintzeKarcher}
Ernst Heintze and Hermann Karcher.
\newblock A general comparison theorem with applications to volume estimates
  for submanifolds.
\newblock {\em Annales scientifiques de l'\'Ecole Normale Sup\'erieure}, Ser.
  4, 11(4):451--470, 1978.

\bibitem[HLS17]{HLS}
Lan-Hsuan Huang, Dan~A. Lee, and Christina Sormani.
\newblock Intrinsic flat stability of the positive mass theorem for graphical
  hypersurfaces of {E}uclidean space.
\newblock {\em J. Reine Angew. Math.}, 727:269--299, 2017.

\bibitem[IKKS04]{IKK}
Tadeusz Iwaniec, Janne Kauhanen, Aleda Kravetz, and Chad Scott.
\newblock The hadamard-scwarz inequality.
\newblock {\em Journal of Function Spaces and Applications}, 2(2):191--215,
  2004.

\bibitem[Pet97]{Petersen-Survey}
Peter Petersen.
\newblock Convergence theorems in {R}iemannian geometry.
\newblock In {\em Comparison geometry ({B}erkeley, {CA}, 1993--94)}, volume~30
  of {\em Math. Sci. Res. Inst. Publ.}, pages 167--202. Cambridge Univ. Press,
  Cambridge, 1997.

\bibitem[PW97]{Petersen-Wei-integral1}
P.~Petersen and G.~Wei.
\newblock Relative volume comparison with integral curvature bounds.
\newblock {\em Geom. Funct. Anal.}, 7(6):1031--1045, 1997.

\bibitem[PW01]{Petersen-Wei-integral2}
Peter Petersen and Guofang Wei.
\newblock Analysis and geometry on manifolds with integral {R}icci curvature
  bounds. {II}.
\newblock {\em Trans. Amer. Math. Soc.}, 353(2):457--478, 2001.

\bibitem[Sor19]{Sormani-ArzAsc}
Christina Sormani.
\newblock Intrinsic flat {A}rzela-{A}scoli theorems.
\newblock {\em Communications in Analysis and Geometry}, 2019.

\bibitem[SW11]{SW-JDG}
Christina Sormani and Stefan Wenger.
\newblock The intrinsic flat distance between {R}iemannian manifolds and other
  integral current spaces.
\newblock {\em J. Differential Geom.}, 87(1):117--199, 2011.

\bibitem[War66]{Warner}
F.~W. Warner.
\newblock Extension of the rauch comparison theorem to submanifolds.
\newblock {\em Transactions of the American Mathematical Society},
  122(2):341--356, 1966.

\bibitem[Yan92a]{DYang-integral2}
Deane Yang.
\newblock Convergence of {R}iemannian manifolds with integral bounds on
  curvature. {I}.
\newblock {\em Ann. Sci. \'{E}cole Norm. Sup. (4)}, 25(1):77--105, 1992.

\bibitem[Yan92b]{DYang-integral3}
Deane Yang.
\newblock Convergence of {R}iemannian manifolds with integral bounds on
  curvature. {II}.
\newblock {\em Ann. Sci. \'{E}cole Norm. Sup. (4)}, 25(2):179--199, 1992.

\bibitem[Yan92c]{DYang-integral1}
Deane Yang.
\newblock Riemannian manifolds with small integral norm of curvature.
\newblock {\em Duke Math. J.}, 65(3):501--510, 1992.

\end{thebibliography}

\end{document}